\documentclass[12pt]{article}
\usepackage{amsmath, amsfonts, amsthm, amssymb, verbatim}
\usepackage{graphicx}
\usepackage[mathcal]{euscript}
\usepackage{amstext}
\usepackage{amssymb,mathrsfs}
\usepackage{bbm}
\usepackage[latin1]{inputenc}
\newcommand{\R}{\mathbb R}
 
 \newcommand{\N}{\mathbb N}

  \newcommand{\E}{\mathbb E}

\newcommand{\PP}{\mathbb P}

\newcommand{\calF} {\ensuremath {\mathcal{F}}}
\newtheorem{theorem}{Theorem}[section]

 \newtheorem{remark}[theorem]{Remark}
\newtheorem{lemma}[theorem]{Lemma}

\newtheorem{definition}[theorem]{Definition}
\newtheorem{hypothesis}[theorem]{Hypothesis}

\begin{document}

\title{  Well-posedness  of  the non-local conservation law  by stochastic
 perturbation. }

\author{ Christian Olivera\footnote{Departamento de Matem\'{a}tica, Universidade Estadual de Campinas, Brazil.
E-mail:  {\sl  colivera@ime.unicamp.br}.
}}

\date{}

\maketitle


\vspace{0.3cm} \noindent {\bf MSC2010 subject classification:} 60H15, %SPDE
 35R60, %SPDE
 35F10, %Initial conditions for 1-order lin DE
 60H30. %applications of Stoch. An.

%\vspace{0.3cm} \noindent {\bf MSC2000 subject classification:} 60H10
%, 35F15,  60H25, 60H15  .

%
%%%%%%%%%%%%%%%%%%%%%%%%%%%%%%%%%%%%%%%%%%%%%%%%%%%%%%%%%%%%%%%%%%%%%%%%%%%%
\begin{abstract}
 Stochastic  non-local conservation law  equation  in the presence of  discontinuous flux functions
is considered in an $L^{1}\cap L^{2}$ setting. The flux function is assumed bounded and integrable
(spatial variable). Our result is to  prove existence and uniqueness of weak solutions. The solution is strong solution in the probabilistic sense. The proofs are constructive  and based on  the method of characteristics (in the presence  of noise), 
It\^o-Wentzell-Kunita formula and commutators. Our results  are new  , to the best of our knowledge, and are  the first nonlinear extension of the seminar paper  \cite{FGP2} where the linear case  was addressed. 
 \end{abstract}

%%%%%%%%%%%%%%%%%%%%%%%%%%%%%%%%%%%%%%%%%%%%%%%%%%%%%%%%%%%%%%%%%%%%%%%%%%%%
%
\maketitle

%\begin{center}
%{\large
%Christian Olivera\footnote{Research  supported  FAEPEX 1324/12, FAPESP 2012/18739-0, 2012/18780-0  . }}\\
%
%\textit{Departamento de
% Matem\'{a}tica, Universidade Estadual de Campinas, \\
% Campinas - SP, Brasil}
%\\  e-mail:  colivera@imeunicamp.br}
%\end{center}

%%%%%%%%%%%%%%%%%%%%%%%%%%%%%%%%%%%%%%%%%%%%%%%%%%%%
\section {Introduction} \label{Intro}

We consider  the conservation law

\begin{equation}\label{con1}
 \left \{
\begin{aligned}
    &\partial_t u(t, x) + Div (F(t,x,u))= 0 \, ,
    \\[5pt]
    &u|_{t=0}=  u_{0} \, .
\end{aligned}
\right .
\end{equation}

Here u is called the conserved quantity while F is the flux. These  type of the equations  express the balance equations of continuum physics, when small
dissipation effects are neglected. A basic example is provided by the equations of
 gas and fluid  dynamics, traffic flow and sedimentation of solid particles in a liquid. The  well-posedness theorems within the class of entropy solutions were
established by Kruzkov, see \cite{Krus}. The selection of the physically relevant solution is based on 
the so called entropy condition that assert that a shock is formed only when the characteristics carry information toward
the shock.

  In  1995 was introduced by Lions, Perthame and 
Tadmor \cite{LPT}  the notions of called kinetic solution and relies on a new equation, the so-called kinetic formulation, that
is derived from the conservation law at hand and that (unlike the original problem)
possesses a very important feature - linearity. The two notions of solution,
i.e. entropy and kinetic, are equivalent whenever both of them exist, nevertheless,
kinetic solutions are more general as they are well defined even in situations when
neither the original conservation law or the corresponding entropy inequalities can
be understood in the sense of distributions. Among other significant references in this direction, let us emphasize the works
of Chen and Perthame \cite{chen}, Perthame \cite{Perth} and  Lions,  Benoit and  Souganidis \cite{Lions4}.

Recently there has been an interest in studying the effect of stochastic forcing on
nonlinear conservation laws  \cite{chen2,Debu,Feng,Hof}. These papers  consider the following stochastic scalar conservation laws 

\[
    du(t, x) + Div (F(u)) dt =  g(u) dW(t,x)  \, ,
\]

 with particular emphasis
on existence and uniqueness questions (well-posedness).

For other hand, in \cite{lions} and \cite{lions2}    Lions, Benoit and Souganidis, 
 introduced the theory of  pathwise  solutions to study the following stochastic conservation law  

\[
    du(t, x) + Div (F(t,x,u)) {\circ}{dz_t} =0  \, ,
\]

where $z_t$ is continuous noise. They defined  a new concept of the solution and proved 
existence and uniqueness via  kinetic formulation. When the  noise $z_t$ is the 
standard Brownian motion the  pathwise  solutions corresponding formally on the Stratonovich interpretation of the equation. See also
Gess and Souganidis in \cite{Gess2}.

The purpose of the present paper is a
contribution to the following general question: can
one hope for uniqueness of  weak solutions  of  some class  of conservation laws with irregular flux
 under stochastic perturbation. We present  the first  positive result. More precisely, we  study the following stochastic conservation law

\begin{equation}\label{SCL}
 \left \{
\begin{aligned}
    &   \partial_t u(t, x) + Div  \big( ( F(t,x,(K\ast u)(x))  + \frac{d B_{t}}{dt}) \cdot  u(t, x)  \big )= 0, 
  \, ,
    \\[5pt]
    &u|_{t=0}=  u_{0} \, .
\end{aligned}
\right .
\end{equation}
Here, $(t,x) \in [0,T] \times \R$, $\omega \in \Omega$ is an element of the probability space $(\Omega, \PP, \calF)$, 
$F:[0,T] \times \R\times \R  \to \R$ , $B_{t}$ is a standard Brownian motion and $K$ is a regular kernel. 
The stochastic integration is to be understood in the Stratonovich sense. The Stratonovich form is the natural one for several reasons
, including physical intuition related to the Wong-Zakai principle.

The novelty  of our results is to  show existence and uniqueness of 
 weak solutions for one-dimensional stochastic nonlocal  conservation law  (\ref{SCL})  when  $F$ has low regularity in the spatial variable. 
Conservation laws with non-local fluxes have appeared recently in the literature, arising naturally
in many fields of application, such as in crowd dynamics (see \cite{Colom} and the
references therein), or in models inspired from biology( see \cite{Carr,Fri}).

In recent years there has been an increasing interest in random influences on PDE. 
The questions of regularizing effects and well-posedness by noise for 
partial differential equations have attracted much interest in recent years. 
The literature on regularization (i.e. improvement on  uniqueness)  by noise is vast and giving a complete
survey at this point would exceed the purpose of this paper. Concerning the case of
transport/continuity  equations with irregular drift, we mention the works \cite{AttFl11,Beck,Fre1,Fre2,FGP2,MNP14,Moli,NO,Ol}.
In particular, we would like to emphasize the papers \cite{Moli} and \cite{Ol}
 since the proofs have  served as an inspiration for some  steps of this work. 

 The proofs are  based in the method of characteristics, some estimation on the flow associated to characteristics,   
in commutators and  stochastic calculus techniques.

\medskip

\subsection{Hypothesis.}

 We denoted $F_1=\partial_{1} F$, $F_3=\partial_{3} F$ and $F_{3,3}=\partial_{3}^{2} F$ In this paper we assume the following hypothesis:

\begin{hypothesis}\label{hyp1}
The flux $F$ satisfies
\begin{equation}\label{cond3-1}
  F\in L^{\infty}([0,T], L^{1}(\R, L^{\infty}(\R))),
\end{equation}

\begin{equation}\label{cond3-3}
  F\in L^{\infty}( [0,T]\times \R  \times  \R),
\end{equation}

\begin{equation}\label{cond3-2}
   F_1\in L^{\infty}([0,T], L^{1}(\R, L^{\infty}(\R))), 
\end{equation}

\begin{equation}\label{cond3-4}
  F_3\in L^{\infty}( [0,T]\times \R  \times  \R),
\end{equation}

\begin{equation}\label{cond3-5}
   F_{3,3}\in L^{2}([0,T], L^{1}(\R, L^{\infty}(\R))), 
\end{equation}

Moreover, the initial condition verifies 
\begin{equation}\label{weight}
  u_0 \in  L^2(\R)\cap L^1(\R) ,
\end{equation}

and the kernel satisfies 

\begin{equation}
  K \in  C_{0}^{\infty}(\R). 
\end{equation}

\end{hypothesis}

\subsection{Notations}First, through of this paper, we fix a stochastic basis with a
one-dimensional Brownian motion $\big( \Omega, \mathcal{F}, \{
\mathcal{F}_t: t \in [0,T] \}, \mathbb{P}, (B_{t}) \big)$. Then, we recall to help the intuition, the following definitions 

$$
\begin{aligned}
\text{Itô:}&  \ \int_{0}^{t} X_s dB_s=
\lim_{n    \rightarrow \infty}   \sum_{t_i\in \pi_n, t_i\leq t}  X_{t_i}    (B_{t_{i+1} \wedge t} - B_{t_i}),
\\[5pt]
\text{Stratonovich:}&  \ \int_{0}^{t} X_s  \circ dB_s=
\lim_{n    \rightarrow \infty}   \sum_{t_i\in \pi_n, t_i\leq t} \frac{ (X_{t_{i+1} \wedge t   } + X_{t_i} ) }{2} (B_{t_{i+1} \wedge t} - B_{t_i}),
\\[5pt]
\text{Covariation:}& \ [X, Y ]_t =
\lim_{n    \rightarrow \infty}   \sum_{t_i\in \pi_n, t_i\leq t} (X_{t_{i+1}\wedge t   } - X_{t_i} )  (Y_{t_{i+1} \wedge t} - Y_{t_i}),
\end{aligned}
$$
where $\pi_n$ is a sequence of finite partitions of $ [0, T ]$ with size $ |\pi_n| \rightarrow 0$ and
elements $0 = t_0 < t_1 < \ldots  $. The limits are in probability, uniformly in time
on compact intervals. Details about these facts can be found in Kunita 
\cite{Ku2}. Also we address from that book, 
Itô's formula, the chain rule for the stochastic integral, 
for any continuous d-dimensional semimartingale $X = (X_1,X_2,\ldots,X_d)$, 
and twice continuously differentiable and real valued function f on   $\mathbb{R}^{d}$.

\subsection{Stochastic flow.}

We consider  $0\leq s\leq t$ and $x\in\mathbb{R}^{d}$, consider the following
stochastic differential equation in $\mathbb{R}^{d}$

\begin{equation}
\label{11}X_{s,t}(x)= x + \int_{s}^{t} b(r, X_{s,r}(x)) \ dr + B_{t}-B_{s}.
\end{equation}
and denote by $X_{t}(x): = X_{0,t}(x), t\in [0,T], x\in \mathbb{R} ^{d}$.

For $m \in \N$ and $0< \alpha < 1$, let us assume the following hypothesis on $b$:
\begin{equation}
\label{REGULCLASS}
    b\in L^{1}((0,T); C_{b}^{m,\alpha}(\mathbb{R}^{d}))
\end{equation}
where $C^{m,\alpha}(\mathbb{R}^{d})$ denotes the class of functions of class $C ^{m}$  on $\mathbb{R}^{d}$ such that the last derivative is H\"older continuous of order $\alpha$.

\noindent It is well known that under condition (\ref{REGULCLASS}), $X_{s,t}(x)$ is a
stochastic flow of $C^{m}$-diffeomorphism (see for example \cite{Chow} and
\cite{Ku2}). Moreover, the inverse flow
 $$Y_{s,t}(x):=X_{s,t}^{-1}(x)$$
 satisfies the
following backward stochastic differential equation%

\begin{equation}
\label{back}Y_{s,t}(x)= x - \int_{s}^{t} b(r, Y_{r,t}(x)) \ dr - (B_{t}-B_{s}).
\end{equation}
for  every $0\leq s\leq t\leq T$. We will denote $Y_{0,t}(x):= Y_{t}(x)$ for every $t\in [0, T], x\in \mathbb{R}^{d}$.

%%%%%%%%%%%%%%%%%%%%%%%%%%%%%%%%%%%%%%%%%%%%%%%%%%%%

\section{Definition and Existence  of Solutions.}

\subsection{Definition of solutions}
\begin{definition}\label{defisoluH}
A stochastic process 
$u\in   L^\infty([0,T],L^{2}( \Omega\times \R)) \cap L^{1}(\Omega\times [0,T]\times \R)\cap L^{\infty}(\Omega\times [0,T], L^{1}(\R))$  is called a  $L^{2}$- weak solution of the Cauchy problem (\ref{SCL}) when: For any $\varphi \in C_0^{\infty}(\R)$, the real valued process $\int  u(t,x)\varphi(x)  dx$ has a continuous modification which is an $\mathcal{F}_{t}$-semimartingale, and for all $t \in [0,T]$, we have $\mathbb{P}$-almost surely
\begin{equation} \label{DISTINTSTR}
\begin{aligned}
    \int_{\R} u(t,x) \varphi(x) dx = &\int_{\R} u_{0}(x) \varphi(x) \ dx 
		\\[5pt]
   	& + \int_{0}^{t} \!\! \int_{\R}   u(s,x)   \, F(s,x,(K\ast u)) \partial_x \varphi(x)  dx ds
		\\[5pt]
   	& 
+ \int_{0}^{t} \!\! \int_{\R}   u(s,x) \ \partial_x \varphi(x) \ dx \, {\circ}{dB_s} \, .
\end{aligned}
\end{equation}
\end{definition}

\begin{remark}\label{lemmaito}
Using the same idea as in Lemma 13 \cite{FGP2}, one can write the problem (\ref{SCL}) in It\^o form as follows, a  stochastic process $u\in   L^\infty([0,T],L^{2}( \Omega\times \R))\cap L^{1}(\Omega\times [0,T]\times \R) \cap L^{\infty}(\Omega\times [0,T], L^{1}(\R))$ is  a $ L^{2}$- weak solution  of the SPDE (\ref{SCL}) iff for every test function $\varphi \in C_{0}^{\infty}(\mathbb{R})$, the process $\int u(t, x)\varphi(x) dx$ has a continuous modification which is a $\mathcal{F}_{t}$-semimartingale and satisfies the following It\^o's formulation

\[
    \int_{\R} u(t,x) \varphi(x) dx = \int_{\R} u_{0}(x) \varphi(x) \ dx	
   	+ \int_{0}^{t} \!\! \int_{\R}   u(s,x)   \, F(s,x,(K\ast u)) \partial_x \varphi(x) \ dx ds
\]
\[
 + \int_{0}^{t} \!\! \int_{\R}   u(s,x) \ \partial_x \varphi(x) \ dx \, dB_s \,  + 
\frac{1}{2} \int_{0}^{t} \!\! \int_{\R}   u(s,x) \ \partial_x^{2} \varphi(x) \ dx \, ds.
\]

\end{remark}

\subsection{Existence.}

The goal of this section is to prove general existence result for
 stochastic conservation law with low regularity of the flux function.

\begin{lemma}
Assume that hypothesis \ref{hyp1} holds. Then there exists $L^2$-weak solutions of the Cauchy problem (\ref{SCL}).
\end{lemma}

\begin{proof}

{\it Step 1: Regularization.}

 Let $\{\rho_n\}_n$ be a family of standard symmetric mollifiers, We  define the family of  regularized coefficients given by

$$
F^{n}(t,.,z) =   (F\ast_x \rho_n )(t,.,z)  
$$

and

$$
u_0^n =   u_0 \ast \rho_n    \,.
$$

Clearly we observe that, for every  $n\in   \mathbb{N}$, any element $F^{n}$, $u_0^n$ are smooth  with bounded derivatives of all orders. We define $u^{1}$ as the solution of the following SPDE

\begin{equation}\label{STE-reg1}
 \left \{
\begin{aligned}
    &d u^1 (t, x) +  Div\bigg(  u^1 (t, x)  \cdot \big(  F^{1}(t,x,(u_{0}^{1}\ast K))  dt +
 \circ d B_{t} \big) \bigg)= 0\, ,
    \\[5pt]
    &u^\varepsilon \big|_{t=0}=  u_{0}^1.
\end{aligned}
\right .
\end{equation}

Moreover, inductively we define

\begin{equation}\label{STE-reg}
 \left \{
\begin{aligned}
    &d u^{n+1}(t, x) +  Div\bigg(  u^{n+1} (t, x)  \cdot \big(  F^{n+1}(t,x,(u^{n}\ast K))  dt +
 \circ d B_{t} \big) \bigg)= 0\, ,
    \\[5pt]
    &u^{n+1} \big|_{t=0}=  u_{0}^{n+1}.
\end{aligned}
\right .
\end{equation}

Following the classical theory of H. Kunita \cite{Ku3} we obtain that

\[
u^{n+1}(t,x) =  u_{0}^{n+1} (Y_t^{n+1}(x))  JY_t^{n+1}(x)
\]

is the unique solution to the regularized equation \eqref{STE-reg}, where $Y_t^{n+1}$ 
is the inverse  to the following stochastic differential equation (SDE):
\begin{equation*}
d X_t^{n+11} = F^{n+1} (t,X_t,(u^{n}\ast K)(t,X_t) ) \, dt + d B_t \, ,  \hspace{1cm}   X_0 = x \,.
\end{equation*}

\bigskip

{\it Step 2: Boundedness.}  We observe that

\begin{equation}\label{l1}
\int_{\R} |u^{n+1}(t,x)| \,  dx  =
\int_{\R}  |u_{0}^{n+1} (Y_t^{n+1}(x))|  JY_t^{n+1}(x) \,  dx  =
 \int_{\R} |u_{0}^{n+1} (y)|  dx\leq C.
\end{equation}

We postpone the following estimation for  the appendix
 
\begin{equation}\label{ja}
\mathbb{E}\bigg[\bigg|\partial_x X_{s,t}^{n}(x)\bigg|^{-1}\bigg] \leq C, 
\end{equation}

where the constant does  not depend on $n$. 

Then  making the change of variables $y=Y_t^{\varepsilon}(x)$ we have that
\begin{align*}
\int_{\R} \E[|u^{n+1}(t,x)|^2]\,  dx & =    \E \int_{\R} |u_{0}^{n+1} (y)|^2  (JX_t^{n+1}(y))^{-1}  dy.
\end{align*}

Now, by inequality (\ref{ja}) we conclude 
\[
\int_{\R} \E[|u^{n+1}(t,x)|^2]\,  dx =
  \int_{\R} |u_{0}^{n+1} (y)|^2  \E  (JX_t^{n+1}(x))^{-1}  dx
	\]
\begin{align}\label{eq0}	
 = \int_{\R} |u_{0}^{n+1} (y)|^2  \E |\partial_x X_{t}^{n+1}(x)|^{-1}   dx \leq C \int_{\R} |u_{0}^{n+1} (y)|^2  dx.
\end{align}

Therefore, the sequence $\{u^{n}\}$ is bounded in $u\in L^2(\Omega\times [0,T]\times \R )\cap L^\infty([0,T],L^{2}( \Omega\times \R))$ . Then  there exists a convergent subsequence, which we denote also by $u^{n}$, such that converge weakly in $L^2(\Omega\times [0,T]\times \R)$ and weak-star in $L^\infty([0,T],L^{2}( \Omega\times \R))$ to some process $u\in L^2(\Omega\times [0,T]\times \R)\cap L^\infty([0,T],L^{2}( \Omega\times \R))$. Since this subsequence is bounded in $L^{1}(\Omega\times [0,T]\times \R)$ we follows that $u^{n}$ converge to one measure $\mu$ and $\mu=u$. From estimation   (\ref{l1}) we deduce that $u\in L^{\infty}(\Omega\times [0,T], L^{1}(\R))$. 

\bigskip
{\it Step 3: Passing to the Limit.}
Now, if $u^{n+1}$ is a solution of \eqref{STE-reg}, it is also a weak solution, that is, for any test function $\varphi\in C_0^{\infty}(\R)$, $u^{n+1}$ verifies  (written in the Itô form):
\begin{align}\label{pass}
\int_{\R} u^{n+1}(t,x) \varphi(x) dx = &\int_{\R} u^{n+1}_{0}(x) \varphi(x) \ dx 
\nonumber\\[5pt] & +
\int_{0}^{t} \!\! \int_{\R}   u^{n+1}(s,x)   \, F^{n+1}(s,x,(u^{n}\ast K)) \partial_x \varphi(x) \ dx ds \nonumber\\[5pt]
& + \int_{0}^{t} \!\! \int_{\R}   u^{n+1}(s,x) \ \partial_x \varphi(x) \ dx \, dB_s \,  + \frac{1}{2} \int_{0}^{t} \!\! \int_{\R}   u^{n+1}(s,x) \ \partial_x^{2} \varphi(x) \ dx \, ds\,.
\end{align}

Now, we observe that $G^{n+1}=u^{n+1}(s,x)   \, F^{n+1}(s,x,(u^{n}\ast K)) $ is uniformly   bounded in 
$L^2(\Omega\times [0,T]\times \R)$. Then  there exists a convergent subsequence, which we denote also by $G^{n}$, such that converge weakly in $L^2(\Omega\times [0,T]\times \R)$  to some process $G\in L^2(\Omega\times [0,T]\times \R)$.

Then  passing  to the limit in equation  (\ref{pass}) along the convergent subsequences found, we have 

\[
\begin{aligned}
    \int_{\R} u(t,x) \varphi(x) dx = &\int_{\R} u_{0}(x) \varphi(x) \ dx
	+ \int_{0}^{t} \!\! \int_{\R}   G(s,x,\omega) \partial_x \varphi(x) \ dx ds
\\[5pt]
    & + \int_{0}^{t} \!\! \int_{\R}   u(s,x) \ \partial_x \varphi(x) \ dx \, dB_s \,  + \frac{1}{2} \int_{0}^{t} \!\! \int_{\R}   u(s,x) \ \partial_x^{2} \varphi(x) \ dx \, ds.
\end{aligned}
\]

 From the last equality we have that $\int_{\R} u(t,x) \varphi(x) dx$ is continuous semimartingale   
 for any test function $\varphi\in C_0^{\infty}(\R)$. Thus  $u\ast K$ is
continuous semimartingale. We observe that  $u^{n+1}\ast K$ converge to $u\ast K$
and  that $ F^{n+1}(s,x,(u^{n}\ast K)) $ strong converge to $F(s,x,(u\ast K)) $ in $L^{2}([0,T]\times \Omega, L_{loc}^{2}(\R))$. Then  passing  to the limit in equation  (\ref{pass}) along the convergent subsequences found, we conclude that 

$
\begin{aligned}
    \int_{\R} u(t,x) \varphi(x) dx = &\int_{\R} u_{0}(x) \varphi(x) \ dx
	+ \int_{0}^{t} \!\! \int_{\R}  u(s,x) F(s,x,(u\ast K)) \partial_x \varphi(x) \ dx ds
\\[5pt]
    & + \int_{0}^{t} \!\! \int_{\R}   u(s,x) \ \partial_x \varphi(x) \ dx \, dB_s \,  + \frac{1}{2} \int_{0}^{t} \!\! \int_{\R}   u(s,x) \ \partial_x^{2} \varphi(x) \ dx \, ds.
\end{aligned}
$
\end{proof}

\section{Uniqueness.}

\subsection{Estimation on the flow.}

Assume  that $u$ is a $L^{2}$-solution of the Cauchy problem (\ref{SCL}), .
Let $\{\rho_\varepsilon\}_\varepsilon$ be a family of standard symmetric mollifiers.
We denoted  $F_{\varepsilon}(t,.,(u\ast K))= (F\ast_{x} \rho_\varepsilon) (t,.,(u\ast K))$. We observe
 that the regularization is done in the second variable. Following the same steps as in the Lemma \ref{lemma pe2}
we obtain:

\begin{lemma}\label{lemma pe}
Assume the hypothesis \ref{hyp1}.  Then for  $T>0$  there exist a  constant 
$C$  such that
\begin{align}\label{eq0}
\mathbb{E}\bigg[\bigg|\partial_x X_{s,t}^{\varepsilon}(x)\bigg|^{-1}\bigg] \leq C,
\end{align}

where C depend  on $\| F\|_{L^{\infty}([0,T], L^{1}(\R, L^{\infty}(\R)))}$, $\| F\|_{L^{\infty}( [0,T]\times \R  \times  \R)}$, 
$\| F_1\|_{L^{\infty}([0,T], L^{1}(\R, L^{\infty}(\R)))}$, $\| F_3\|_{L^{\infty}( [0,T]\times \R  \times  \R)}$ and 
$\| F_{3,3}\|_{L^{2}([0,T], L^{1}(\R, L^{\infty}(\R)))}$.
\end{lemma}

\begin{remark}The same results is valid for the backward flow $Y_{s,t}$  since it is  solution of the same SDE driven by 
the drifts $-b$.
\end{remark}

%%%%%%%%%%%%%%%%%%%%%%%%%%%%%%%%%%%%% Unicidade %%%%%%%%%%%%%%%%%%%%%%%%%%%%%%

\subsection{Main result.}

In this section, we shall present a uniqueness theorem
for the SPDE (\ref{SCL}).  We pointed that similar  arguments was used 
in previous works \cite{Moli} ,  \cite{Ol} for stochastic  linear continuity equation.

\begin{theorem}\label{uni2}
Under the conditions of hypothesis \ref{hyp1}, uniqueness holds for  $L^{2}$- weak solutions of the Cauchy problem (\ref{SCL}) in the following sense: if $u,v$ are $L^{2}$- weak solutions with the same initial data $u_{0}\in  L^2(\R)\cap L^{1}(\R)$, then  $u= v$ almost everywhere in $ \Omega  \times [0,T] \times \R$.
\end{theorem}

\begin{proof}

{\it Step 1:  Two solutions.}  Let $v$  and $w$ are two   $L^{2}$- weak solutions with initial conditions equal to
$u_{0}$. We denoted $u=v-w$. Then 
$u$ verifies

\[
\begin{aligned}
    \int_{\R} u(t,x) \varphi(x) dx = & 
\\[5pt]
    & 	+ \int_{0}^{t} \!\! \int_{\R}   u(s,x)   \, F(s,x,(K\ast v)) \partial_x \varphi(x) \ dx ds\\[5pt]  &
+ \int_{0}^{t} \!\! \int_{\R}   u(s,x) \ \partial_x \varphi(x) \ dx \, {\circ}{dB_s} 	\nonumber\\[5pt]
    & 	+ 
		\int_{0}^{t} \!\! \int_{\R}   w(s,x)   \,  \big(F(s,x,(K\ast w))-F(s,x,(K\ast v)) \big)  \partial_x \varphi(x) \ dx ds
\end{aligned}
\]

{\it Step 2:  Primitive of the solution.} 
We set

\[
V(t,x)=\int_{-\infty}^{x} u(t,y) \ dy.
\]

 We observe that $\partial_x V(t,x)=u(t,x)$  belong to $L^{2}( \Omega\times[0, T]\times \mathbb{R} )$. Now, we consider a  nonnegative smooth cut-off function $\eta$ supported on the ball of radius 2 and such that  $\eta=1$ on the ball of radius 1. For any $R>0$, we introduce the rescaled functions $\eta_R (\cdot) =  \eta(\frac{.}{R})$.

\bigskip

 For all  test functions  $\varphi\in C_0^{\infty}(\R)$ we obtain 
\[
 \int_{\R} V(t,x) \varphi(x) \eta_R (x)  dx = - \int_{\R} u(t,x)   \theta(x) \eta_R (x)  dx
-\int_{\R} V(t,x)   \theta(x) \partial_x \eta_R (x)  dx\,,
\]

\noindent where $\theta(x) =\int_{-\infty}^{x}  \varphi(y)   \ dy$. By  definition  of $L^{2}$-solutions , taking as test function $ \theta(x) \eta_R (x)$ we get 

\begin{align}\label{DISTINTSTRTR}
    \int_{\R} & V(t,x) \ \eta_R (x) \varphi(x) dx =    \nonumber\\[5pt]
    &- \int_{0}^{t} \!\! \int_{\R}   \partial_x V(s,x)   \, F(s,x,(K\ast v)) \eta_R (x) \varphi(x) \ dx ds \nonumber\\[5pt]
    & - \int_{0}^{t} \!\! \int_{\R}   \partial_x V(s,x) \  \eta_R (x) \varphi(x) \ dx \, {\circ}{dB_s} 
		\nonumber\\[5pt]   &
		- 	\int_{0}^{t} \!\! \int_{\R}   \partial_x V(s,x)  F(s,x,(K\ast v))  \,  \partial_x \eta_R (x) \theta(x)  \ dx ds\nonumber\\[5pt]
    &- \int_{0}^{t} \!\! \int_{\R}   \partial_x V(s,x) \  \partial_x \eta_R (x) \theta(x)  \ dx \, {\circ}{dB_s}-\int_{\R} V(t,x)   \theta(x) \partial_x \eta_R (x)  dx \nonumber\\[5pt]  &
		- \int_{0}^{t} \!\! \int_{\R}   w(s,x)   \,  \big(F(s,x,(K\ast w))-F(s,x,(K\ast v)) \big)    \eta_R (x) \varphi(x) \ dx ds 
		\nonumber\\[5pt]  &
		- \int_{0}^{t} \!\! \int_{\R}   w(s,x)   \,  \big(F(s,x,(K\ast w))-F(s,x,(K\ast v)) \big)   \partial_x \eta_R (x) \theta(x) \ dx ds
\end{align}

We observe that

\[
\int_{0}^{t} \!\! \int_{\R}   \partial_x V(s,x) \  \partial_x \eta_R (x) \theta(x)  \ dx \, {\circ}{dB_s}\rightarrow 0,
\]

\[
\int_{\R} V(t,x)   \theta(x) \partial_x \eta_R (x)  dx \rightarrow 0,
\]

\[
\int_{0}^{t} \!\! \int_{\R}   w(s,x)   \,  \big(F(s,x,(K\ast w))-F(s,x,(K\ast v)) \big)   \partial_x \eta_R (x) \theta(x) \ dx ds
\rightarrow 0 ,
\]

\[
\int_{0}^{t} \!\! \int_{\R}   \partial_x V(s,x)  F(s,x,(K\ast v))  \,  \partial_x \eta_R (x) \theta(x)  \ dx ds
\rightarrow 0
\]
as $R\rightarrow \infty$. Passing to the limit  in equation  (\ref{DISTINTSTRTR})  we have that

		\[
		\int_{\R} V(t,x) \varphi(x) dx =  
		\]
		
\[
	- \int_{0}^{t} \!\! \int_{\R}   \partial_x V(s,x)  F(s,x,(K\ast v)) \,   \varphi(x) \ dx ds
	\]
\[
- \int_{0}^{t} \!\! \int_{\R}   \partial_x V(s,x) \   \varphi(x) \ dx \, {\circ}{dB_s} 
\]
\[
- \int_{0}^{t} \!\! \int_{\R}   w(s,x)   \,  \big(F(s,x,(K\ast w))-F(s,x,(K\ast v)) \big)    \varphi(x) \ dx ds 
\]
\bigskip

{\it Step 3: Smoothing.}
Let $\{\rho_{\varepsilon}(x)\}_\varepsilon$ be a family of standard symmetric mollifiers. For any $\varepsilon>0$ and $x\in\R^d$ we use $\rho_\varepsilon(x-\cdot)$ as test function, then we deduce 
$$
\begin{aligned}
      \int_{\R} V(t,y) \rho_\varepsilon(x-y) \, dy  = & \\[5pt]
    &- \int_{0}^{t}  \int_{\R} \big( F(s,y,(K\ast v))  \partial_y V(s,y)  \big)  \rho_\varepsilon(x-y) \ dy ds
		\\[5pt]
    & -  \int_{0}^{t} \!\! \int_{\R} \partial_y V(s,y) \, \rho_\varepsilon(x-y)  \, dy \circ dB_s
		\\[5pt] &
		- \int_{0}^{t} \!\! \int_{\R}   w   \,  \big(F(s,y,(K\ast w))-F(s,y,(K\ast v)) \big)    \rho_\varepsilon(x-y) \ dy ds 
\end{aligned}
$$

We denote   $V_\varepsilon(t,.)= (V\ast_{x} \rho_\varepsilon)(t,.)$, $F_\varepsilon(t,.,(K\ast v)(.)) = (F \ast_{x} \rho_\varepsilon)(t,.)$ and
$(FV)_\varepsilon(t,.)= (F.V\ast_{x} \rho_\varepsilon)(t,.)$. Thus we have

\[
    V_{\varepsilon}(t,x) + \int_{0}^{t} F_{\epsilon}(s,x,(K\ast v))  \partial_x V_{\varepsilon}(s,x) \,  ds   +  \int_{0}^{t}   \partial_{x}  V_{\varepsilon}(s,x) \, \circ dB_s
\]		
	
\[         
=    \int_{0}^{t} \big(\mathcal{R}_{\epsilon}(V,F) \big) (x,s) \,  ds  
\]
\[
-\int_{0}^{t} \!\! \int_{\R}   w(s,x)   \,  \big(F(s,y,(K\ast w))-F(s,y,(K\ast v)) \big)    \rho_\varepsilon(x-y) \ dy ds 
\]

\noindent where we denote
$ \mathcal{R}_{\epsilon}(V,F)  = F_\varepsilon \ \partial_x V_\varepsilon  -  (F\partial_x V)_\varepsilon  $.

\bigskip

{\it Step 4: Method of Characteristics.} Now, we consider the flow

\begin{equation*}
d X_t^{\epsilon} = F_{\epsilon}(t, X_t^{\epsilon},(K\ast v)(t,  X_t^{\epsilon})) \, dt + d B_t \, ,  \hspace{1cm}   X_0 = x \,,
\end{equation*}

Applying the It\^o-Wentzell-Kunita formula   to $ V_{\varepsilon}(t,X_{t}^{\epsilon})$
, see Theorem 8.3 of \cite{Ku2}, we have

\[
   V_{\varepsilon}(t,X_{t}^{\epsilon})  = \int_{0}^{t} \big(\mathcal{R}_{\epsilon}(V,F) \big) (X_s^{\epsilon},s)  ds  
\]

\[         
-\int_{0}^{t} \!\! \int_{\R}   w   \,  \big(F(s,y,(K\ast w))-F(s,y,(K\ast v)) \big)    \rho_\varepsilon(X_s^{\epsilon}-y) \ dy ds 
\]

Then, considering that $X_{t}^{\epsilon}=X_{0,t}^{\epsilon}$ and $Y_{t}^{\epsilon}=Y_{0,t}^{\epsilon}=(X_{0,t}^{\epsilon})^{-1}$ 
we deduce that 

\[
   V_{\varepsilon}(t,x)  =\int_{0}^{t} \big(\mathcal{R}_{\epsilon}(V,F) \big) (Y_{t-s}^{\epsilon},s)  ds  
\]

\[         
-\int_{0}^{t} \!\! \int_{\R}   w   \,  \big(F(s,x,(K\ast w))-F(s,x,(K\ast v)) \big)    \rho_\varepsilon(Y_{t-s}^{\epsilon}-y) \ dy ds 
\]

{\it Step 5: Localization.} Now, we consider a  nonnegative smooth cut-off function $\eta$ supported on the ball of radius 2 and such that  $\eta=1$ on the ball of radius 1. For any $R>0$, we introduce the rescaled functions $\eta_R (\cdot) =  \eta(\frac{.}{R})$. From the last step
we have

\[
  \int_{\R}  |V_{\varepsilon}(t,x)| \eta_R(x) dx  
	\]
	
	\[
	\leq \int_{0}^{t} \ \int_{\R}  |\big(\mathcal{R}_{\epsilon}(V,F) \big) (Y_{t-s}^{\epsilon},s)|  \eta_R(x) dx   ds  
\]

\begin{equation}\label{loca}         
+ \int_{0}^{t} \!\! \int_{\R} \eta_R(x) \int_{\R}  | w|   \,  |F(s,y,(K\ast w))-F(s,y,(K\ast v)) |   \rho_\varepsilon(Y_{t-s}^{\epsilon}-y) \ dy  dx ds 
\end{equation}

\bigskip

{\it Step 6: Convergence of the commutator I.}   Now, we observe that $\mathcal{R}_{\epsilon}(V,F)$ converge to zero in
$L^{2}([0,T]\times \Omega\times \R )$. In fact, we  get that

\[
  (F \ \partial_x V)_{\varepsilon} \rightarrow F \ \partial_x V \ in \ L^{2}([0,T] \times  \Omega \times \R ).
\]

\noindent Moreover,  we have 

\[
F_{\epsilon} \rightarrow F \  \ in \ L^{1}([0,T]\times \Omega, L_{loc}^{1}(\R))
\]

 and

\[
\partial_x V_{\epsilon} \rightarrow   \partial_x V \ in \ L^{2}([0,T]\times \Omega\times \R ).
\]

Then by the dominated convergence theorem we obtain

\[
F_{\epsilon}   \partial_x V_{\varepsilon} \rightarrow F \ \partial_x V \ in \ L^{2}([0,T]\times  \Omega \times \R ).
\]

\bigskip

{\it Step 7: Convergence of the conmutator II.}  We observe that 

\[
	\int_{0}^{t}  \int   |\big(\mathcal{R}_{\epsilon}(V,F) \big) (Y_{t-s}^{\epsilon},s)| \eta_R(x) \ \, dx \   ds
	\]

\[
=	\int_{0}^{t}  \int    |\big(\mathcal{R}_{\epsilon}(V,F) \big) (s,x)| \   JX_{t-s}^{\epsilon}  \eta_R(X_{t-s}^{\epsilon}) \ \, dx \   ds .
\]

 By   H\"older's inequality we obtain 

\[
	\E \bigg|\int_{0}^{t}  \int \bigg(\mathcal{R}_{\epsilon}(V,F) \bigg) (x,s) \   JX_{s,t}^{\epsilon}  \eta_R(X_{t-s}^{\epsilon}) \ \, dx \   ds \bigg|
	\]
	
	\[
	\leq   	 \bigg(\E \int_{0}^{t}  \int |\big(\mathcal{R}_{\epsilon}(V,F) \big) (x,s)|^{2} \  \, dx \   ds \bigg)^{\frac{1}{2}}
  \bigg(\E \int_{0}^{t}  \int  | JX_{s,t}^{\epsilon} \eta_R(X_{t-s}^{\epsilon})|^{2} \ \, dx \ ds \bigg)^{\frac{1}{2}}
\]

\noindent  From step 6 we follow

\[
 \bigg(\E \int_{0}^{t}  \int |\big(\mathcal{R}_{\epsilon}(V,F) \big) (x,s)|^{2} \  \, dx \   ds \bigg)^{\frac{1}{2}}\rightarrow 0.
\]

From lemma \ref{lemma pe} we deduce 

\[
\bigg(\E \int_{0}^{t}  \int  | JX_{s,t}^{\epsilon} \eta_R(X_{t-s}^{\epsilon})|^{2} \ \, dx \ ds \bigg)^{\frac{1}{2}}
=\bigg(\E \int_{0}^{t}  \int  | JY_{s,t}^{\epsilon}|^{-1} |\eta_R(x)|^{2} \ \, dx \ ds \bigg)^{\frac{1}{2}}
\]
\[
\leq C  \bigg( \int |\eta_R(x)|^{2} \ \, dx \  \bigg)^{\frac{1}{2}},
\]

{\it Step 8: Conclusion  .}  From step 5  we have

\[
  \int_{\R}  |V_{\varepsilon}(t,x)| \eta_R(x) dx  
	\]
	
	\[
	\leq \int_{0}^{t} \ \int_{\R}  |\big(\mathcal{R}_{\epsilon}(V,F) \big) (Y_{t-s}^{\epsilon},s)|  \eta_R(x) dx   ds  
\]

\[
+ \int_{0}^{t} \!\! \int_{\R} \eta_R(x)  \int_{\R}  | w(s,x)|   \,  |F(s,y,(K\ast w))-F(s,y,(K\ast v)) |   \rho_\varepsilon(Y_{t-s}^{\epsilon}-y) \ dy dx  ds
ds
\] 

Now we observe that 

\[
 \int_{0}^{t} \!\! \int_{\R} \eta_R(x)  \int_{\R}  | w(s,x)|   \,  |F(s,x,(K\ast w))-F(s,x,(K\ast v)) |   \rho_\varepsilon(Y_{t-s}^{\epsilon}-y) \ dy dxds
ds
\]

\[
\leq  C \| w \|_{L^{1}(\R)} \int_{0}^{t} \!\!   \int_{\tilde{K}} |V(s,z)| dz  ds
\] 

\noindent where $\tilde{K}$ is one compact set. 
If we take $R$ sufficiently large we have

\[
  \int_{\R}  |V_{\varepsilon}(t,x)| \eta_R(x) dx  
	\]
	
	\[
	\leq \int_{0}^{t} \ \int_{\R}  |\big(\mathcal{R}_{\epsilon}(V,F) \big) (Y_{t-s}^{\epsilon},s)|  \eta_R(x) dx   ds  
\]

\[
+ \  C \| w \|_{L^{1}(\R)} \int_{0}^{t} \!\!   \int \eta_R(x)  |V(s,z)| dz  ds
\]

Taking the limit as $\epsilon$ converge to zero we deduce 

\[
  \int_{\R}  |V(t,x)| \eta_R(x) dx  
	\]

\[
\leq  C \| w \|_{L^{1}(\R)} \int_{0}^{t} \!\!   \int \eta_R(x)  |V(s,z)| dz  ds
\]

By Gronwall lemma we deduce   that $V=0$. Then we have  $u=0$.

\end{proof}

\section{Appendix}

\begin{lemma}\label{lemma pe2}
Assume the hypothesis \ref{hyp1}.  Then for  $T>0$  there exist a  constant 
$C$  such that

\begin{align}\label{eq0}
\mathbb{E}\bigg[\bigg|\partial_x X_{s,t}^{n}(x)\bigg|^{-1}\bigg] \leq C,
\end{align}
where C depend on  $\| F\|_{L^{\infty}([0,T], L^{1}(\R, L^{\infty}(\R)))}$, $\| F\|_{L^{\infty}( [0,T]\times \R  \times  \R)}$,  
$\| F_1\|_{L^{\infty}([0,T], L^{1}(\R, L^{\infty}(\R)))}$, $\| F_3\|_{L^{\infty}( [0,T]\times \R  \times  \R)}$ and 
$\| F_{3,3}\|_{L^{2}([0,T], L^{1}(\R, L^{\infty}(\R)))}$.

\end{lemma}

\begin{proof}
For simplicity we assume $s=0$. 

{\it Step 1: Regularization.}

We consider the SDE :
\begin{equation*}
d X_t^{{n+1}} = F^{n+1} (t,X_t,(u^{n}\ast K)(t,X_t^{n+1})) \, dt + d B_t \, ,  \hspace{1cm}   X_0 = x \,.
\end{equation*}

We observe that  $\partial_x X_{t}^{n+1}$ verifies 

\[
\partial_x X_{t}^{{n+1}}=\exp\bigg\{ \int_{0}^{t}  \partial_{x} \big( F^{n+1}(s,x,(u^{n}\ast K))\big)(s,X_t^{n+1}) \  ds  \bigg\}. 
\]

{\it Step 2: Semimartingale representation.} By definition of solution we get  that  $(K\ast u^{n})(t,x)$ satisfies

\[
    \int_{\R} u^{n}(t,y) K(x-y) dy =  \int_{\R} u_{0}(y) K(x-y) \ dy
\]	
\[
	+ \int_{0}^{t} \!\! \int_{\R}   u^{n}(s,y)   \, F^{n+1}(s,y,(u^{n-1}\ast K)) \partial_y K(x-y) \ dy ds
+ \int_{0}^{t} \!\! \int_{\R}   u^{n}(s,y) \ \partial_y K(x-y) \ dy \, dB_s \,  
\]
\[
+ \frac{1}{2} \int_{0}^{t} \!\! \int_{\R}   u^{n}(s,y) \ \partial_y^{2} K(x-y) \ dy \, ds.
\]

Then by It\^o formula we have

\[
    F^{n+1}(t,x,(K\ast u^{n})) =   F^{n+1}(0,x,(K\ast u_{0})) +  \int_{0}^{t} \!\!     \,  F_1^{n+1}(s,x,(u^{n}\ast K)) \ ds
\]	
\[
	+ \int_{0}^{t} \!\!    F_3^{n+1}(s,x,(K\ast u^{n}))   \int_{\R}   u^{n}(s,y)   \, F^{n+1}(s,y,(K\ast u^{n-1})) \partial_y K(x-y) \ dy ds
	\]
	
\begin{equation}\label{itoF2}
+ \int_{0}^{t} \!\!   F_{3}^{n+1}(s,x,(K\ast u^{n}))    \int_{\R}   u^{n}(s,y) \ \partial_y K(x-y) \ dy \,  {\circ}{dB_s} \,  
\end{equation}

In order to write the last equality in  It\^o formulation we have to calculate the covariation

\[
\big[   F_3^{n+1}(s,x,(K\ast u^{n})(x))    \int_{\R}   u^{n}(s,x) \ \partial_y K(x-y) \ dy, B_s \big].
\]

Now, we obtain

\[
    F_3^{n+1}(t,x,(K\ast u^{n})) =   F_3^{n+1}(0,x,(K\ast u_{0})) +  \int_{0}^{t} \!\!     \, F_{3,1}^{n+1}(s,x,(K\ast u^{n})) \ ds
\]	
\[
	+ \int_{0}^{t} \!\!    F_{3,3}^{n+1}(s,x,(K\ast u^{n}))   \int_{\R}   u(s,y)   \, F^{n+1}(s,y,(K\ast u^{n-1})) \partial_y K(x-y) \ dy ds
	\]
	
\[
+ \int_{0}^{t} \!\!   F_{3,3}^{n+1}(s,x,(K\ast u^{n}))    \int_{\R}   u^{n}(s,y) \ \partial_y K(x-y) \ dy \,  {\circ}{dB_s} \,  
\]

and

\[
    \int_{\R} u^{n}(t,x) \partial_y K(x-y) dy = \int_{\R} u_{0}(y)  \partial_y K(x-y) \ dy
\]	
\[
	+ \int_{0}^{t} \!\! \int_{\R}   u^{n}(s,y)   \, F^{n+1}(s,y,(K\ast u^{n-1})) \partial_{y,y} K(x-y) \ dy ds
\]	
\[	
+ \int_{0}^{t} \!\! \int_{\R}   u^{n}(s,y) \ \partial_{y,y} K(x-y) \ dy \, {\circ}{dB_s}  \,  
\]

We set  $ u^{n,k}(t,x):= (K \ast u^{n})(t,x)$. From the  It\^o formula for the product of two semimartingales we obtain

\[
    F_3^{n+1}(t,x,(K\ast u^{n}))   \partial_x u^{n,k}(t,x) =   F_3^{n+1}(0,x,(K\ast u_{0}))  \partial_x u_0^{k}(t,x) 
\]
\[			
		+ \int_{0}^{t} \!\!    \partial_x u^{n,k}(s,x)  \,  F_{3,1}^{n+1}(s,x,(K\ast u^{n})) \ ds
\]	
\[
	+ \int_{0}^{t}  \partial_x u^{n,k}(s,x)    F_{3,3}^{n+1}(s,x,(K\ast u^{n}))   
	\int_{\R}   u(s,y)   \, F^{n+1}(s,y,(K\ast u^{n-1})) \partial_y K(x-y) \ dy ds
	\]
	
\[
+ \int_{0}^{t} \!\!  \partial_x u^{n,k}(s,x) F_{3,3}^{n+1}(s,x,(K\ast u^{n}))   \partial_x u^{n,k}(s,x) \,  {\circ}{dB_s} \,  
\]

\[
	+ \int_{0}^{t} \!\!    F_3^{n+1}(s,x,(K\ast u^{n}))  \int_{\R}   u^{n}(s,y)   \, F^{n+1}(s,y,(K\ast u^{n-1})) \partial_{y}^{2} K(x-y) \ dy ds
\]	
\[
+ \int_{0}^{t} \!\!    F_3^{n+1}(s,x,(K\ast u^{n}))    \partial_{x}^{2} u^{n,k}(s,x)  \, {\circ}{dB_s}  \,  
\]

Applying covariation in the last equality we deduce

\[
\big[   F_{3}^{n+1}(s,x,(K\ast u^{n}))    \int_{\R}   u^{n}(s,x) \ \partial_y K(x-y) \ dy, B_s \big]
\]

\[
 =\int_{0}^{t} \!\!  \partial_x u^{n,k}(s,x)  F_{3,3}^{n+1}(s,x,(K\ast u^{n}))   \partial_x u^{k}(s,x) \,  ds \,  
\]

\begin{equation}\label{cova2}
 + \int_{0}^{t}      F_3^{n+1}(s,x,(K\ast u^{n}))    \partial_{x,x} u^{k}(s,x) \ ds.      
\end{equation}

From formulas (\ref{itoF2}) and (\ref{cova2}) we obtain 

\[
    F^{n+1}(t,x,(K\ast u^{n})) =   F^{n+1}(0,x,(K\ast u_{0})) +  \int_{0}^{t} \!\!     \,  F_1^{n+1}(s,x,(K\ast u^{n})) \ ds
\]	
\[
	+ \int_{0}^{t} \!\!    F_{3}^{n+1}(s,x,(K\ast u^{n}))   \int_{\R}   u^{n}(s,y)   \, F^{n+1}(s,y,(K\ast u^{n-1})) \partial_y K(x-y) \ dy ds
	\]
\[	
 \int_{0}^{t} \!\!   F_{3}^{n+1}(s,x,(K\ast u^{n}))    \int_{\R}   u^{n}(s,y) \ \partial_y K(x-y) \ dy \,  dB_s \,  
\]

\[
 +\int_{0}^{t} \!\!  \partial_x u^{n,k}(s,x)  F_{3,3}^{n+1}(s,x,(K\ast u^{n}))   \partial_x u^{n,k}(s,x) \,  ds \,  
\]

\[
 + \int_{0}^{t}      F_3^{n+1}(s,x,(K\ast u^{n}))    \partial_{x}^{2} u^{n,k}(s,x) \ ds.      
\]

Integrating the last equality  we get

\[
     \int_{-\infty}^{z} F^{n+1}(t,x,(K\ast u^{n}))  dx =   \int_{-\infty}^{z} F^{n+1}(0,x,(K\ast u_{0}) dx 
\]		
	
		\[+  
		\int_{0}^{t}     \, \int_{-\infty}^{z}  F_{1}^{n+1}(s,x,(K\ast u^{n})) \ dx ds
\]	
\[
	+ \int_{0}^{t} \!\!  \int_{-\infty}^{z}   F_{3}^{n+1}(s,x,(K\ast u^{n}))   \int_{\R}   u^{n}(s,y)   \, F^{n+1}(s,y,(K\ast u^{n-1})) \partial_y K(x-y) \ dy dx  ds
	\]
\[	
 +\int_{0}^{t} \!\! \int_{-\infty}^{z}    F_{3}^{n+1}(s,x,(K\ast u^{n}))    \int_{\R}   u^{n}(s,y) \ \partial_y K(x-y) \ dy  dx \,  dB_s \,  
\]

\[
 +\int_{0}^{t} \!\! \int_{-\infty}^{z}  \partial_x u^{n,k}(s,x) \partial_{3,3} F^{n+1}(s,x,(K\ast u^{n}))   \partial_x u^{n,k}(s,x) \, dx  ds \,  
\]

\[
 + \int_{0}^{t}    \int_{-\infty}^{z}   F_3^{n+1}(s,x,(K\ast u^{n}))    \partial_{x}^{2} u^{n,k}(s,x) \ dx ds.      
\]

{\it Step 3: It\^o-Wentzell-Kunita formula.}  Applying the It\^o-Wentzell-Kunita formula   to $ \int_{-\infty}^{X_{t}^{n+1}} F(t,x,(K\ast u^{n})(x))  dx$ , see Theorem 8.3 of \cite{Ku2}, we deduce

\[
     \int_{-\infty}^{X_{t}^{n+1}} F^{n+1}(t,x,(K\ast u^{n}))  dx =   \int_{-\infty}^{x} F^{n+1}(0,x,(K\ast u_{0})) dz 
\]		
	\[	
	+	\int_{0}^{t} \!\!     \, \int_{-\infty}^{X_{s}^{n+1}}  F_{1}^{n+1}(s,x,(K\ast u^{n})) \ dx ds
\]	
\[
	+ \int_{0}^{t} \!\!  \int_{-\infty}^{X_{s}^{n+1}}   F_{3}^{n+1}(s,x,(K\ast u^{n}))   \int_{\R}   u^{n}(s,y)   \, F^{n+1}(s,y,(K\ast u^{n-1})) \partial_y K(x-y) \ dy dx  ds
	\]
\[	
 \int_{0}^{t}    \int_{-\infty}^{X_{s}^{n+1}}    F_{3}^{n+1}(s,x,(K\ast u^{n}))    \int_{\R}   u^{n}(s,y) \ \partial_y K(x-y) \ dy  dx \,      \,  dB_s \,  
\]

\[
 +\int_{0}^{t} \!\! \int_{-\infty}^{X_{s}^{n+1}}  \partial_x u^{n,k}(s,x)  F_{3,3}^{n+1}(s,x,(K\ast u^{n}))   \partial_x u^{n,k}(s,x) \, dx  ds \,  
\]

\[
 + \int_{0}^{t}    \int_{-\infty}^{X_{s}^{n+1}}   F_3^{n+1}(s,x,(K\ast u^{n}))    \partial_{x,x} u^{n,k}(s,x) \ dx ds
\]
\[
+ \int_{0}^{t}  F^{n+1}(s,X_{s}^{n+1},(K\ast u^{n})(s,X_{s}^{n+1})) F^{n+1}(s,X_{s}^{n+1},(K\ast u^{n})(s,X_{s}^{n+1}))     ds 
\]
\[
+  \int_{0}^{t}  F^{n+1}(s,X_{s}^{n+1},(K\ast u^{n})(s,X_{s}^{n+1}))        dB_s  
\]

\[
+ \frac{1}{2}  \int_{0}^{t}  \partial_{x}\big( F^{n+1}(s,x,(K\ast u^{n} ))\big)(s,X_s^{n+1})    ds  
\]
\[
+ \int_{0}^{t}    F_3^{n+1}(s,X_{s}^{n+1},(K\ast u^{n}))   \partial_{x}u^{n,k}(s,X_{s}^{n+1})      ds
\]

{\it Step 4: Boundedness.}

We have that

\[
\|  \int_{-\infty}^{X_{t}^{n+1}} F^{n+1}(t,x,(K\ast u^{n}))  dx\|_{L^{\infty}(\Omega\times [0,T]\times \R)} \leq \|F\|_{L^{\infty}([0,T],L^{1}( \R,L^\infty(\R)))} ,
\]

\[
\| \int_{-\infty}^{x} F^{n+1}(0,x,(K\ast u_{0})) dx  \|_{L^{\infty}(\Omega\times [0,T]\times \R)}\leq C   \|F\|_{L^{\infty}([0,T],L^{1}( \R,L^\infty(\R)))},
\]

\[
\|  \int_{0}^{t} \!\!     \, \int_{-\infty}^{X_{s}^{n+1}}  F_1^{n+1}(s,x,(K\ast u^{n})) \ dx ds \|_{L^{\infty}(\Omega\times [0,T]\times \R)} 
\leq C   \|F_1\|_{L^{\infty}([0,T],L^{1}( \R,L^\infty(\R)))},
\]

\[
\|    \int_{0}^{t} \!\!  \int_{-\infty}^{X_{s}^{n+1}}  \ F_{3}^{n+1}(s,x,(K\ast u^{n}))   \int_{\R}   u^{n}(s,y)   \, F^{n+1}(s,y,(K\ast u^{n-1})) \partial_y K(x-y) \ dy dx  ds      \|_{L^{\infty}(\Omega\times [0,T]\times \R)}
\]

\[
\leq C  \|  F  \|_{L^{\infty}([0,T]\times \R \times\R)}    \|  u^{n} \|_{L^{\infty}([0,T]\times \Omega,L^{1}(\R))}    
\|F_3\|_{L^{\infty}([0,T],L^{1}( \R,L^\infty(\R)))},
\]
\[
\leq C  \|  F  \|_{L^{\infty}([0,T]\times \R \times\R))}    \|F_2\|_{L^{\infty}([0,T],L^{1}( \R,L^\infty(\R)))},
\]
\[
\|   \int_{0}^{t} \!\! |\int_{-\infty}^{X_{s}^{n+1}}   F_3^{n+1}(s,x,(K\ast u^{n})(x))    \int_{\R}   u^{n}(s,y) \ \partial_y K(x-y) \ dy  dx \, |^{2} ds  \|_{L^{\infty}(\Omega\times [0,T]\times \R)} \leq C
\]

\[  
\leq C  \|  u^{n} \|_{L^{\infty}([0,T]\times \Omega,L^{1}(\R))}^{2}    \|F_3\|_{L^{2}([0,T],L^{1}( \R,L^\infty(\R)))},
\]

\[  
\leq C     \|F_3\|_{L^{2}([0,T],L^{1}( \R,L^\infty(\R)))},
\]

\[
\| \int_{0}^{t} \!\! \int_{-\infty}^{X_{s}^{n+1}}  \partial_x u^{n,k}(s,x) F_{3,3}^{n+1}(s,x,(K\ast u^{n}))   \partial_x u^{n,k}(s,x) \, dx  ds   \|_{L^{\infty}(\Omega\times [0,T]\times \R)}
\]
\[
\leq C  \|  u^{n} \|_{L^{\infty}([0,T]\times \Omega,L^{1}(\R))}^{2}   \|F_{3,3}\|_{L^{1}( [0,T]\times \R,L^\infty(\R)))},
\]
\[
\leq C    \|F_{3,3}\|_{L^{1}( [0,T]\times \R,L^\infty(\R)))},
\]

\[
\| \int_{0}^{t}    \int_{-\infty}^{X_{s}^{n+1}}   F_3^{n+1}(s,x,(K\ast u^{n}))    \partial_{x,x} u^{n,k}(s,x) \ dx ds  \|_{L^{\infty}(\Omega\times [0,T]\times \R)}
\]
\[
\leq C  \|  u^{n} \|_{L^{\infty}([0,T]\times \Omega,L^{1}(\R))}^{2}    \|F_3\|_{L^{1}([0,T]\times \R, L^{\infty}(\R) )} ,
\]

\[
\leq C    \|F_3\|_{L^{1}([0,T]\times \R, L^{\infty}(\R) )} ,
\]

\[
\|  \int_{0}^{t}  F^{n+1}(s,X_{s}^{n+1},(K\ast u^{n})(s,X_{s}^{n+1}))  F^{n+1}(s,X_{s},(K\ast u^{n})(s,X_{s}^{n+1}))    ds  \|_{\infty} \leq   \|F\|_{L^{\infty}( [0,T]\times \R \times \R)}^{2} ,
\]

\[
\|  \int_{0}^{t}   F_3^{n+1}(s,X_{s}^{n+1},(K\ast u^{n})(s,X_{s}^{n+1}))   \partial_{x}u^{k}(s,X_{s}^{n+1})      ds \|_{L^{\infty}(\Omega\times [0,T]\times \R)}
\]
\[
 \leq   \|  u^{n} \|_{L^{\infty}([0,T]\times \Omega,L^{1}(\R))} \| F_3\|_{L^{\infty}( [0,T]\times \R \times \R)}.
\]

\[
 \leq   \| F_3\|_{L^{\infty}( [0,T]\times \R \times \R)}.
\]

{\it Step 5: Conclusion.}

From step 3 we have

\[
-\int_{0}^{t}  \partial_{x}\big( F^{n+1}(s,x,(K\ast u^{n} ))\big)(s,X_s^{n+1})    ds
\]

\[
=  2 \int_{0}^{t}  F^{n+1}(s,X_{s}^{n+1},(K\ast u^{n})(s,X_{s}^{n+1}))  dB_s  
\]

\[
- 4 \int_s^{t}   |F^{n+1}(s,X_{s},(K\ast u^{n}))|^{2}(s,X_{s})ds 
\]

\[
+ 2 \int_{0}^{t}    \int_{-\infty}^{X_{s}^{n+1}}    F_{3}^{n+1}(s,x,(K\ast u^{n}))    \int_{\R}   u^{n}(s,y) \ \partial_y K(x-y) \ dy  dx \,      \,  dB_s \,  
\]
\[
- 4 \int_{0}^{t}   | \int_{-\infty}^{X_{s}^{n+1}}    F_{3}^{n+1}(s,x,(K\ast u^{n}))    \int_{\R}   u^{n}(s,y) \ \partial_y K(x-y) \ dy  dx \, |^{2}     \,  ds \,  
\]

\begin{equation}\label{EI}
+ I_{3}= I_{1}+I_{2} + I_{3}
\end{equation}

where 

\[
I_{1}=  2 \int_{0}^{t}  F^{n+1}(s,X_{s}^{n+1},(K\ast u^{n})(s,X_{s}^{n+1}))  dB_s  
\]

\[
- 4 \int_s^{t}   |F^{n+1}(s,X_{s},(K\ast u^{n}))|^{2}(s,X_{s})ds,
\]

and 

\[
I_{2} = 2 \int_{0}^{t}    \int_{-\infty}^{X_{s}^{n+1}}    F_{3}^{n+1}(s,x,(K\ast u^{n}))    \int_{\R}   u^{n}(s,y) \ \partial_y K(x-y) \ dy  dx \,      \,  dB_s \,  
\]
\[
- 4 \int_{0}^{t}   | \int_{-\infty}^{X_{s}^{n+1}}    F_{3}^{n+1}(s,x,(K\ast u^{n}))    \int_{\R}   u^{n}(s,y) \ \partial_y K(x-y) \ dy  dx \, |^{2}     \,  ds \,  
\]

amd 

\[
    I_{3}(t,x) =- 2 \int_{-\infty}^{X_{t}^{n+1}} F^{n+1}(t,x,(K\ast u^{n}))  dx + 2  \int_{-\infty}^{x} F^{n+1}(0,x,(K\ast u_{0})) dz 
\]

	\[	
	+	2 \int_{0}^{t} \!\!     \, \int_{-\infty}^{X_{s}^{n+1}}  F_{1}^{n+1}(s,x,(K\ast u^{n})) \ dx ds
\]	

\[
	+ 2 \int_{0}^{t} \!\!  \int_{-\infty}^{X_{s}^{n+1}}   F_{3}^{n+1}(s,x,(K\ast u^{n}))   \int_{\R}   u^{n}(s,y)   \, F^{n+1}(s,y,(K\ast u^{n-1})) \partial_y K(x-y) \ dy dx  ds
	\]

\[
 + 2 \int_{0}^{t} \!\! \int_{-\infty}^{X_{s}^{n+1}}  \partial_x u^{n,k}(s,x)  F_{3,3}^{n+1}(s,x,(K\ast u^{n}))   \partial_x u^{n,k}(s,x) \, dx  ds \,  
\]

\[
 + 2 \int_{0}^{t}    \int_{-\infty}^{X_{s}^{n+1}}   F_3^{n+1}(s,x,(K\ast u^{n}))    \partial_{x,x} u^{n,k}(s,x) \ dx ds
\]

\[
+  6 \int_{0}^{t} | F^{n+1}(s,X_{s}^{n+1},(K\ast u^{n})(s,X_{s}^{n+1}))|^{2}     ds 
\]

\[
+ 2 \int_{0}^{t}    F_3^{n+1}(s,X_{s}^{n+1},(K\ast u^{n}))   \partial_{x}u^{n,k}(s,X_{s}^{n+1})      ds
\]

\[
+ 4 \int_{0}^{t}   | \int_{-\infty}^{X_{s}^{n+1}}    F_{3}^{n+1}(s,x,(K\ast u^{n}))    \int_{\R}   u^{n}(s,y) \ \partial_y K(x-y) \ dy  dx \, |^{2}     \,  ds .  
\]

\noindent We set
\[
\mathcal{E}\bigg(\int_0^t 4 F^{n+1}(s,X_{s},(K\ast u^{n}))(s,X_{s})dB_s\bigg)
\]
\[
=\exp\bigg\{\int_0^t   4 F^{n+1}(s,X_{s},(K\ast u^{n}))(s,X_{s}) dB_s-8 \int_s^{t}   |F^{n+1}(s,X_{s},(K\ast u^{n}))|^{2}(s,X_{s})ds \bigg\}, 
\]
and

\[
\mathcal{E}\bigg(\int_0^t   \int_{-\infty}^{X_{s}}   4 F_3^{n+1}(s,x,(K\ast u^{n}))    \int_{\R}   u^{n}(s,y) \ \partial_y K(x-y) \ dy  dx   dB_s\bigg)
\]
\[
=\exp\bigg\{\int_0^t    \int_{-\infty}^{X_{s}}   4 F_3^{n+1}(s,x,(K\ast u^{n})    \int_{\R}   u^{n}(s,y) \ \partial_y K(x-y) \ dy  dx    dB_s
\]
\[
-8 \int_s^{t}   |\int_{-\infty}^{X_{s}}    F_3^{n+1}(s,x,(K\ast u^{n}))    \int_{\R}   u^{n}(s,y) \ \partial_y K(x-y) \ dy  dx|^{2} ds \bigg\}. 
\]

 From (\ref{EI}) and   inequalities in step 4  we obtain

\[
\E\bigg[\bigg|\frac{d X_{t}}{dx}(x)\bigg|^{-1}\bigg ] \leq C \ \E\bigg[ \exp(I_{1}+I_{2})    \bigg ]
\]

Finally by  H\"older inequality we deduce 

\begin{align}
\E\bigg[\bigg|\frac{d X_{t}}{dx}(x)\bigg|^{-1}\bigg ] & 
\leq 
\nonumber\\&   \leq C \mathcal{E}\bigg(\int_0^t 4 F^{n+1}(s,X_{s},(K\ast u^{n})(s,X_{s})))dB_s\bigg)^{1/2} 
\nonumber\\ & \times
\mathcal{E}\bigg(\int_0^t   \int_{-\infty}^{X_{s}} 4  F_3^{n+1}(s,x,(K\ast u^{n}))    \int_{\R}   u^{n}(s,y) \ \partial_y K(x-y) \ dy  dx   dB_s\bigg)^{1/2}
\end{align}

Finally we observe  that the  processes  

\[
\mathcal{E}\bigg(\int_0^t  4F^{n+1}(s,X_{s},(K\ast u^{n+1}))(s,X_{s})))dB_s\bigg)
\]

and

\[
\mathcal{E}\bigg(\int_0^t   \int_{-\infty}^{X_{s}}  4   F_{3}^{n+1}(s,x,(K\ast u^{n+1}))    \int_{\R}   u^{n+1}(s,y) \ \partial_y K(x-y) \ dy  dx   dB_s\bigg)
\]

 are martingales with expectation equal to one.  From this  we conclude our lemma.
\end{proof}

\section{Final remarks. }

\begin{remark}(Linear case) Now suppose that $F(t,x,z)=b(x)$.  If $b\in L^{1}\cap L^{\infty}$ then
it satisfies the hypothesis 1. Thus, we have the uniqueness for the  stochastic continuity equation 
with irregular drift. Then our result is the  nonlinear extension of the theory of regularization   
by noise for transport/continuity equation initiated by Flandoli, Gubinelli and Priola
in \cite{FGP2}. 

We pointed  of according  to the theory of Diperna-Lions (see \cite{DL}) the uniqueness  of 
 deterministic transport/continuity equation holds when $b$ has $W^{1,1}$ spatial regularity together with a condition of boundedness on the divergence. The theory has been generalized by L. Ambrosio \cite{ambrisio} to the case of only $BV$ regularity for b instead of $W^{1,1}$.
 We refer the readers to two excellent summaries in \cite{ambrisio2} and \cite{lellis}.
\end{remark}

\begin{remark}
The main our  tool in order to have estimations on the derivative of the flow was the It\^o-Wentzell-Kunita formula. However, 
it is possible  only to apply this  formula for compositions of semimartingales. In order to generalized our result for more general $F$ we have in mind to work in the context of the theory of  stochastic calculus  via regularization. This calculus  was introduced by 
  by F. Russo and P. Vallois ( see \cite{RVSem} as general reference ) and it have been studied and developed by
many authors.  In the papes of F.Flandoli and F. Russo  and R. Coviello and  F.Russo 
they obtain a  It\^o-Wentzell-Kunita formula for more general process, see \cite{Flan} and \cite{Covie}. 
We also mention the recent extension of the It\^o-Wentzell-Kunita formula by R. Duboscq  and  A. Reveillac 
in \cite{Dobu}. 
\end{remark}
 
\begin{remark}  We pointed that 
multiplicative noise as the one used in the  stochastic conservation law (\ref{SCL}) is not enough to improve 
the regularity of solutions of the following stochastic Burgers equation
$$
    \partial_t u(t, x,) +  \partial_x u(t, x)  \big( u(t,x) + \frac{d B_{t}}{dt} \big) = 0 \, .
$$
Indeed, for this equation one can observe the appearance of shocks in finite time, just as for the deterministic Burgers equation. 
We address the reader to  \cite{Flanlect} for a more detailed discussion of this topic.
\end{remark}

\begin{remark} Finally we point after the writing of this paper appeared in  arxiv the paper  of Gess and Maurelli \cite{Gess}  where  the authors consider stochastic scalar conservation laws with spatially inhomogeneous
flux. Assuming  low regularity of the flux function with respect to its
 spatial variable  they proved uniqueness of  stochastic kinetic  entropy solutions 
when are not necessarily uniqueness in the  corresponding deterministic scalar conservation law.
Our result is in one-dimension but we prove uniqueness in the class of weak solutions 
and we assume very low regularity on the flux functions (in spatial variable). 
\end{remark}

%%%%%%%%%%%%%%%%%%%%%%%%
\section*{Acknowledgements}
%%%%%%%%%%%%%%%%%%%%%%%%
    Christian Olivera  is partially supported by FAPESP by the grants 2017/17670-0 and 2015/07278-0 and by CNPq
		by the grant 426747/2018-6.

%%%%%%%%%%%%%%%%%

\end{document}